\documentclass[a4paper, 11pt]{article}

\usepackage{amsmath,amssymb ,amstext ,mathrsfs, extarrows}
\usepackage[utf8]{inputenc}
\usepackage{float}
\usepackage{dsfont}
\usepackage{yfonts}
\usepackage{amsfonts}
\usepackage{geometry}
\usepackage{stmaryrd}
\usepackage{mathabx}
\usepackage{hyperref}
\usepackage{tikz}
\usetikzlibrary{cd}
\usepackage{xfrac}
\usepackage[arrow, matrix, curve]{xy}
\usepackage{titling}
\predate{}
\postdate{}
\usepackage[english,german]{babel}

\usepackage[explicit]{titlesec}
\usepackage{lipsum}

\usepackage{amsthm}
\newtheorem{thm}{Theorem}[section]
\newtheorem*{thm*}{Theorem}
\newtheorem{lem}[thm]{Lemma}
\newtheorem{prop}[thm]{Proposition}
\newtheorem{cor}[thm]{Corollary}

\theoremstyle{definition}
\newtheorem{defi}[thm]{Definition}

\theoremstyle{remark}
\newtheorem{rem}[thm]{Remark}

\DeclareMathOperator{\MOD}{Mod}

\newcommand{\Z}{\mathbb{Z}}
\newcommand{\Q}{\mathbb{Q}}

\renewcommand{\1}{\mathds{1}}

\DeclareMathOperator{\GL}{GL}

\DeclareMathOperator{\Ind}{Ind}
\DeclareMathOperator{\cind}{c-Ind}
\DeclareMathOperator{\Hom}{Hom}

\DeclareMathOperator{\cts}{cts}

\DeclareMathOperator{\Ban}{Ban}
\DeclareMathOperator{\adm}{adm}
\newcommand{\Badm}{\Ban^{\adm}}

\DeclareMathOperator{\Ker}{Ker}

\newcommand{\inv}{^{-1}}
\DeclareMathOperator{\locan}{la}
\newcommand{\la}{{\locan}}
\DeclareMathOperator{\image}{im}
\DeclareMathOperator{\lc}{lc}
\DeclareMathOperator{\co}{co}
\DeclareMathOperator{\id}{id}

\title{\vspace{-2cm}\textbf{Continuous group cohomology with coefficients in locally analytic vectors of admissible $ \Q_p $-Banach space representations}}
\author{Paulina Fust}
\date{}
\begin{document}
	\maketitle
	\selectlanguage{english}
	\begin{abstract}
		We show that the continuous cohomology groups of a $ p $-adic reductive group with coefficients in the locally analytic vectors of an admissible $ \Q_p $-Banach space representation are homeomorphic to those with coefficients in the Banach space representation itself. Moreover, we deduce that the canonical topologies on those continuous cohomology groups are Hausdorff and are the uniquely determined finest locally convex topologies.
	\end{abstract}
	
	\section{Introduction} \label{Intro}
	Let $ G $ be a $ p $-adic reductive group and $ L $ a finite field extension of $ \Q_p $. For an admissible unitary $ L $-Banach space representation $ \Pi $ of $ G $, we study the continuous group cohomology of $ G $ with coefficients in the subrepresentation $ \Pi^\la $ of $ \Pi $ consisting of locally analytic vectors. In the case, where $ G $ is a compact group, Rodrigues Jacinto and Rodríguez Camargo proved in \cite[Corollary 1.6]{rodrigues2022solid}, that these cohomology groups $ H^i(G,\Pi^\la) $ are isomorphic to the cohomology groups with coefficients in the representation $ \Pi $ itself. We use this result as key ingredient for our proof of the following Theorem:
	
	\begin{thm}[Proposition \ref{prop: adm and la same}]\label{intro thm adm=la}
		Let $ \Pi \in \Badm_G(L) $ be an admissible $ L $-Banach space representation of $ G $ and let $ \Pi^\la $ be the subspace of locally analytic vectors in $ \Pi $, equipped with its natural topology. The inclusion $ \Pi^\la\hookrightarrow \Pi $ induces natural isomorphisms between continuous group cohomology groups\[ H^i(G,\Pi^\la)\cong H^i(G,\Pi),\ \forall i\ge 0. \]
	\end{thm}
	
	Here, we consider the subrepresentation $ \Pi^\la $ equipped with the topology defined by Emerton in Section 3 of \cite{emerton-locallyanalytic}, which is finer than the subspace topology coming from the inclusion $ \Pi^\la\subseteq\Pi $. To prove Theorem \ref{intro thm adm=la}, we construct for any continuous representation of $ G $ on an $ L $-vector space $ V $, a spectral sequence converging to the continuous cohomology of $ G $ with coefficients in $ V $, whose first page terms are certain continuous cohomology groups of stabilizers of facets in the Bruhat--Tits building of $ G $. This allows us to deduce Theorem \ref{intro thm adm=la} from the statement for compact groups.
%	 prove the following theorem which compares the continuous cohomology groups for admissible Banach space representations with the ones of their subrepresentation of locally analytic vectors. 
%	
%
%	
%	The result has been proved in \cite[Corollary 1.6]{rodrigues2022solid} in the case where $ G $ is a compact group. We use this as key ingredient for our proof of Theorem \ref{intro thm adm=la}.
%	
	In Corollary 6.10 in \cite{fust2021}, we proved that the continuous cohomology groups $ H^i(G,\Pi) $, for $ \Pi $ an admissible $ L $-Banach space representation, are finite dimensional. This combined with Theorem \ref{intro thm adm=la} implies the following corollary.
	
	\begin{cor}[Corollary \ref{cor H(G,Pi^la) finite}] \label{intro cor H(G,Pi^la) finite}
		For every $ \Pi \in \Badm_G(L) $, the continuous cohomology groups $ H^i(G,\Pi^\la) $ are finite dimensional over $ L $ for all $ i\ge 0 $.
	\end{cor}

	For a topological $ G $-module $ V $, the spaces of continuous $ i $-cochains $ C^i(G,V) $ are equipped with the compact-open topology. We thus get a quotient topology on the continuous cohomology groups $ H^i(G,V) $, called the canonical topology. Moreover, if the $ G $-module $ V $ is a locally convex $ L $-vector space, then we show that the compact-open topology on $ C^i(G,V) $, as well as the canonical topology on $ H^i(G,V) $ is locally convex. The canonical topology is in general not Hausdorff. However, Theorems \ref{intro cor H(G,Pi^la) finite} and \ref{intro thm adm=la} allow us to apply a criterion for the cohomology groups being Hausdorff in \cite{CW} to deduce the following.
	
	\begin{cor}[Corollary \ref{cor Hausdorff}] \label{intro cor hausdorff}
		Suppose that $ G $ is a $ p $-adic reductive group and $ \Pi $ an admissible $ L $-Banach space representation of $ G $. Then for all $ i\ge 0 $, the canonical topologies on the continuous cohomology groups $ H^i(G,\Pi) $ and $ H^i(G,\Pi^\la) $ are Hausdorff.
	\end{cor}

	We use Corollary \ref{intro cor hausdorff} to deduce that the isomorphisms in Theorem \ref{intro thm adm=la} are in fact homeomorphisms.

	\begin{cor}[Corollary \ref{cor finest locally convex}]
		The canonical topology on the continuous cohomology groups $ H^i(G,\Pi^\la) $ and $ H^i(G,\Pi) $ is the finest locally convex topology. In particular, for all $ i\ge 0 $, we have homeomorphisms \[ H^i(G,\Pi^\la)\cong H^i(G,\Pi). \]
	\end{cor}

	\begin{rem}
		The finest locally convex topology on the finite dimensional cohomology groups $ H^i(G,\Pi) $ and $ H^i(G,\Pi^\la) $ can be described more explicitly: Let $ V $ be a finite dimensional $ L $-vector space. Then there is only one way to equip $ V $ with a topology making it a locally convex Hausdorff $ L $-vector space, i.e.~the finest locally convex topology. Namely, it is the topology defined by the norm $ \Vert \sum_{i=1}^n\lambda_i e_i \Vert := \max_{1\le i \le n}\vert \lambda_i\vert$, where $ e_1,\dots, e_n $ is an $ L $-basis of $ V $ (Proposition 4.13 \cite{schneider2001nonarchimedean}).
	\end{rem}
	
		\textit{Acknowledgements}: This work was written as part of my PhD thesis under the supervision of Vytautas Pa\v{s}k\={u}nas. I want to thank him for his support and Nicolas Dupré and Jonas Franzel for their helpful comments and suggestions. This work was funded by the DFG Graduiertenkolleg 2553. 
%	I also want to thank Gabriel Dospinescu, Jessica Fintzen and Lennart Gehrmann for their help on my previous work that led to this one. This work was funded by the DFG Graduiertenkolleg 2553. 
	\section{Continuous group cohomology}
	For coherence, we briefly recall the definition of continuous group cohomology for a locally profinite group $ G $ and collect some basic properties needed in later sections.
	
	Following \cite{CW}, we define the continuous group cohomology as follows:
	Let $ V $ be a topological $ G $-module, i.e.~a topological abelian group $ V $ with a $ G $-action such that $ G $ acts on $ V $ via group automorphisms and the map $ G\times V\rightarrow V $ is continuous. 
	Then we can define the cochain complex $$ C^n(G,V):= C(G^{n+1},V) :=\{f:G^{n+1}\rightarrow V \text{ continuous}\},$$ with differentials $ d^n:C^{n}(G,V)\rightarrow C^{n+1}(G,V) $, defined by $$
	d^nf(g_0,\dots,g_{n+1})=\sum_{i=0}^{n+1}(-1)^i f(g_0,\dots , \widehat{g_i},\dots, g_{n+1}).$$
	
	We endow the spaces $ C^n(G,V) $ with the compact-open topology. More explicitly, a basis of the compact-open topology is given by finite intersections of sets of the form $ \Omega(K,U)=\{f\in C^n(G,V)\vert\ f(K)\subset U\} $, for compact subsets $ K\subset G^{n+1} $ and open subsets $ U\subset V $. By defining a $ G $-action via $ (gf)(g_0,\dots,g_n)=gf(g^{-1}g_0,\dots ,g^{-1}g_n) $, for $ g,g_0,\dots,g_n\in G $ and $ f\in C^n(G,V) $, these will define topological $ G $-modules. Furthermore, there is a continuous $ G $-equivariant injection $ V\hookrightarrow C(G,V) $, defined by $ v\mapsto [g\mapsto v] $.\\
	
	For a closed subgroup $ H\le G $ and a topological $ G $-module $ V $, define the \textit{induced representation} to be $$ \Ind_H^GV:=\{f\in C(G,V) \mid f(hg)=hf(g)\ \forall h\in H,\ \forall g\in G\} ,$$ 
	with the subspace topology induced from $ \Ind^G_HV\subseteq C(G,V) $ and $ G $-action $ gf(g'):=f(g'g) $, for $ g,g'\in G $. Similarly, the \textit{compact induction} is defined as $$ \cind_H^GV:=\{f\in \Ind_H^G V\mid  \text{ the support of } f \text{ is compact modulo }H\}  $$ with the same $ G $-action.
	
	\begin{lem} \label{Cn=C0}
		For a topological $ G $-module $ V $, there are homeomorphisms of $ G$-modules \begin{equation*}
		C^{n+1}(G,V) \cong C^0(G,C^n(G,V)) \text{ and } C^0(G,V)\cong \Ind^G_1V,
		\end{equation*}
		for all $ n\ge 0 $.
	\end{lem}
	
	\begin{proof}
		By \cite[X.3.4 Corollaire 2]{Bourbaki}, the map \begin{align*}
		\phi:	C^{n+1}(G,V) &\rightarrow C^0(G,C^n(G,V)) \\
		f&\mapsto \phi(f),
		\end{align*}
		where $ \phi(f)(g_0)(g_1,\dots,g_n)=f(g_0,\dots,g_n) $, is a homeomorphism. It is straightforward to check that this is $ G $-equivariant, making it an isomorphism of topological $ G $-modules.
		
		The assignment $ f\mapsto  [g\mapsto gf(g^{-1})]$ defines both maps $ C^0(G,V)\rightarrow \Ind^G_1V  $ and its inverse. One can check that both maps are continuous and $ G $-equivariant, hence the modules are homeomorphic.
		%For the first homeomorphism, see \cite[X.3.4 Corollaire 2]{Bourbaki}.
		%		
		%		The assignment $ f\mapsto  [g\mapsto gf(g^{-1})]$ defines both maps $ C^0(G,V)\rightarrow \Ind^G_1V  $ and its inverse. One can check that both maps are continuous and $ G $-equivariant, hence the modules are homeomorphic.
	\end{proof}
	
	\begin{lem}\label{Frobenius}
		For two topological $ G $-modules $ V $ and $ W $, one has the following isomorphism: \[\Hom_G^{\cts}(V,\Ind^G_1W)\cong \Hom^{\cts}(V,W), \]
		where $ \Hom^{\cts}(V,W) $ are continuous group homomorphisms and on the left hand side, we take continuous $ G $-equivariant group homomorphisms.
	\end{lem}
	
	\begin{proof}
		\cite[Lemma 2]{CW}.
	\end{proof}
	
	\begin{lem}
		The complex $ 0\rightarrow V\rightarrow C^\bullet(G,V)  $ is an exact complex of $ G $-modules.
	\end{lem}
	
	\begin{proof}
		\cite[Lemma 2.3]{fust2021}
%		Let $ C^\bullet $ be the complex with $ C^{-1}=V $, $ C^i=0 $, for $ i\le -2 $ and $ C^i=C^i(G,V) $ for $ i\ge 0 $. 
%		To prove the exactness of this complex, we construct a cochain homotopy 
%		$ s^n:C^n\rightarrow C^{n-1} $, $ n\in\Z $ between $ id_{C^\bullet} $ and the zero map on $ C^\bullet $. Define $ (s^nf)(g_1,\dots , g_n):= f(1,g_1,\dots,g_n) $ for $ f\in C^n(G,V) ,\ n\ge 0$ and $ s^n=0 $ for $ n\le -1 $. Then $ s^nf :G^{n}\rightarrow V$ is a continuous map, since it is the composition of the continuous maps $ f $ and $ \{1\}\times G^n\hookrightarrow G\times G^n $. Moreover, one can easily check that it satisfies $ s^{n+1}d^n+d^{n-1}s^n=id_{C^n} $ for all $ n $. Hence we have found the desired homotopy and the complex is exact.
	\end{proof}
	
	\begin{defi}
		The $ i^\text{th} $ \textit{continuous group cohomology group} of $ G $ with coefficients in $ V $ is defined to be $ H^i(G,V):= H^i(C^\bullet(G,V)^G) $.
	\end{defi} 
	
	A homomorphism of topological $ G $-modules $ \phi:V\rightarrow W $ is said to be a \emph{strong morphism}, if the induced morphisms $ \Ker \phi\rightarrow V $ and $ V/\Ker\phi\rightarrow W $ each have a continuous left inverse in the category of topological abelian groups. And if one has a short exact sequence of topological $ G $-modules \[ 0\rightarrow V\rightarrow W \rightarrow U\rightarrow 0, \] in which all the morphisms are strong, then the induced sequence \[ 0\rightarrow C^n(G,V)\rightarrow C^n(G,W)\rightarrow C^n(G,U)\rightarrow 0  \] is again strong and then induces a long exact sequence in cohomology. 
	
	For example, any short exact sequence\[ 0\rightarrow V\rightarrow W \rightarrow U\rightarrow 0, \]of topological $ G $-modules is automatically strong if $ U $ is discrete. Therefore it  induces a long exact sequence in cohomology. 
	
	\begin{defi}
		We say that a topological $ G $-module $ V $ is \textit{acyclic} (for the continuous group cohomology), if  \[ H^i(G,V)=\begin{cases}
		V^G, \text{ if } i=0\\
		0,\text{ otherwise}.
		\end{cases}  \]
	\end{defi}
	
	\begin{rem}
		\begin{enumerate}
			\item The $ G $-modules $ C^n(G,V) $ are acyclic for the continuous group cohomology. Hence, the complex $ 0\rightarrow V\rightarrow C^\bullet(G,V) $ is an acyclic resolution of $ V $. (cf. \cite[p. 201]{CW}.)
			\item A topological $ G $-module $ V $ is said to be \emph{continuously injective}, if for every strong $ G $-injection $ U\hookrightarrow W $ and for every morphism of topological $ G $-modules $ \phi:U\rightarrow V $, the morphism $ \phi $ extends to a morphism from $ W $ to $ V $.
			An example of continuously injective $ G $-modules are the modules $ C^n(G,V) $ (cf. \cite[p. 201]{CW}).
%			\item If we equip a smooth representation $ \pi\in\Mod $ with the discrete topology, this will give a topological $ G $-module and we can define the continuous group cohomology groups $ H^i(G,\pi) $ with coefficients in $ \pi $.
		\end{enumerate}
	\end{rem}

	\begin{lem}\label{computes H(K)}
		For an open subgroup $ K\le G $ and a topological $ G $-module $ V $, the $ G $-modules $
		C^n(G,V) $ are acyclic for the continuous group cohomology $ H^\bullet(K,-) $. 
		
		Moreover, the cohomology of the complex $ C^\bullet(G,V)^K $ is $ H^\bullet(K,V) $.
	\end{lem}
	
	\begin{proof}
		A proof is given in \cite[Proposition 4 (a)]{CW}. The idea is as follows: since the quotient $ G/K $ is discrete, one has a continuous section $ s:G/K\rightarrow G $ of the natural projection $ G\rightarrow G/K $ with which one can define the $ K $-invariant homeomorphism $ K\times G/K\rightarrow G,\ (h,gK)\mapsto s(gK)h $. Therefore, one gets an isomorphism of $ K $-modules $$ C(G,V)\cong C(K\times G/K,V)\cong C(K,C(G/K,V)) $$ showing that $ C(G,V) $ is acyclic as a $ K $-module.
		
		To conclude that the complex $ C^\bullet(G,V)^K $ computes the cohomology of $ K $, one also has to use that the resolution $ V\hookrightarrow C^\bullet(G,V) $ is a strong resolution of acyclic $ K $-modules and hence can be used to compute the cohomology of $ K $ (\cite[Proposition 1]{CW}). 
	\end{proof}
	
	\section{Locally analytic vectors of admissible Banach space representations} \label{sec: locally analytic shit}

	We fix the following notation: let $ G $ be a $ p $-adic reductive group and $ L $ be a finite field extension of $ \Q_p $. Since $ G $ is a $ p $-adic reductive group, it is in particular a locally $ \Q_p $-analytic group. We want to study the continuous cohomology groups of $ G $ with coefficients in the representation of locally analytic vectors in an admissible Banach space representation. For this, we use a spectral sequence, which we will first construct in a more general setting.
	
	\subsection{Construction of a spectral sequence} \label{subsec spectral sequ}

%	In Section \ref{subsec: p-adic red grps}, we used a resolution of the trivial representation $ \1\in \MOD_G^{\sm}(\O/\varpi^n) $ by representations which are compactly induced from compact-mod-center stabilizer subgroups of facets in the Bruhat--Tits building of $ G $, to construct a certain spectral sequence converging to the continuous cohomology of $ G $ with coefficients in a smooth representation. In the following, we will generalize this construction to obtain a similar spectral sequence
In order to construct the spectral sequence, we use a resolution of the trivial representation in terms of certain compactly induced representations from stabilizer subgroups of vertices in the Bruhat--Tits building. Following the notations of \cite{schneider1997representation}, we denote by $ X $ the reduced Bruhat--Tits building associated to $ G $ and for any $ q\ge 0 $ we denote by $ X_q $ the set of all $ q $-dimensional facets in $ X $. Moreover, for a $ q $-facet $ F\in X_{q} $, we define $ P_F^\dagger $ to be the $ G $-stabilizer of $ F $. In \cite{schneider1997representation}, the authors use the fact that the Bruhat--Tits tree is a contractible space to construct a $ G $-equivariant resolution of $ \Z $. We use a reformulation of this resolution as described in Section 6.4 of \cite{fust2021}: 

\begin{equation}\label{resolution of 1}
\dots\rightarrow\bigoplus_{F\in R_1}\cind^G_{P_F^\dagger}\delta_F \rightarrow\bigoplus_{F\in R_0}\cind^G_{P_F^\dagger}\delta_F\rightarrow \Z \rightarrow 0.
\end{equation}

In this section, we use the resolution \eqref{resolution of 1} to construct a spectral sequence of the form
\begin{equation}\label{spectral sequence general}
E_1^{i,j}=\bigoplus_{F\in R_i}H^j(P_F^\dagger, V\otimes_L \delta_F\inv)\Rightarrow H^{i+j}(G,V),
\end{equation}
where $ \delta_F $ are characters on $ P_F $, taking values in $ \{\pm 1\} $.
 
%	 for any $ G $-representation $ V $ on an $ L $-vector space such that the $ G $-action makes it a topological $ G $-module. In the following, we fix such a representation $ V $. %Recall that we denoted by $ X_q $ the set of $ q $-dimensional facets in the Bruhat--Tits building $ X $ of $ G $ and $ R_q $ was a fixed set of representatives of $ X_q $ modulo the action of $ G $. For such a facet $ F \in X_q$, $ P_F^\dagger $ denotes the $ G $-stabilizer of $ F $ and $ \delta_F:P_F^\dagger\rightarrow \{\pm 1 \} $ is a character.
%	
%	We use the resolution \eqref{resolution of 1} that was constructed in Section \ref{subsec: p-adic red grps}\[ \bigoplus_{F\in R_\bullet}\cind_{P_{F}^\dagger}^G\delta_F\rightarrow \Z\rightarrow 0. \]
	
	A priori, the complex \eqref{resolution of 1} is a resolution by smooth $ G $-representations on free $ \Z $-modules, but by tensoring with $ L $, we obtain a resolution of the trivial representation by smooth $ G $-representations on $ L $-vector spaces. \[ \bigoplus_{F\in R_\bullet}\cind_{P_{F}^\dagger}^G\delta_F\rightarrow \1\rightarrow 0. \]
	
	By abuse of notation, we now denote by $ \delta_F $ the representation \[ \delta_F:P_F^\dagger\rightarrow\{\pm 1\}\subset L^\times \] on the one-dimensional $L  $-vector space. 
	
	Since the compactly induced representations in this resolution are smooth, we can give them the structure of topological $ G $-modules by equipping them with the discrete topology.
	
	Now consider the double complex \begin{equation}\label{double complex general}
			C^{i,j}:=\Hom_{G}(\bigoplus_{F\in R_i}\cind_{P_{F}^\dagger}^G\delta_F,C^j(G,V)),\ i,j\ge 0.
	\end{equation}
	
		Note that since we equipped the compactly induced representations with the discrete topology, all the homomorphisms in $ C^{i,j} $ are automatically continuous and we can also write it as $ C^{i,j}=\Hom^{\cts}_{G}(\bigoplus_{F\in R_i}\cind_{P_{F}^\dagger}^G\delta_F,C^j(G,V)). $
		
	This induces two spectral sequences with the same limit term. One is coming from a horizontal filtration and the other from the vertical one. To write these down explicitly, we need to understand the cohomology of the rows and of the columns of our double complex \eqref{double complex general}. We start by studying the rows:
	
	Fixing an $ i\ge 0$, we can write \begin{equation}\label{eq Cî=C^0}
		 C^i(G,V)\cong C^0(G,C^{i-1}(G,V))\cong \Ind_1^GC^{i-1}(G,V)
	\end{equation} (cf. Lemma \ref{Cn=C0}), with the convention that $ C^{-1}(G,V)=V $. In particular, we have isomorphisms $ C^i(G,V)^G\cong (\Ind_1^GC^{i-1}(G,V))^G\cong C^{i-1}(G,V) $. The isomorphism \eqref{eq Cî=C^0} allows us then to apply Frobenius reciprocity for topological $ G $-modules (cf. Lemma \ref{Frobenius}). We obtain \begin{align*}
	C^{\bullet,i} &= \Hom^{\cts}_{G}(\bigoplus_{F\in R_\bullet}\cind_{P_{F}^\dagger}^G\delta_F,C^i(G,V))\\
		&\cong \Hom^{\cts}_{L}(\bigoplus_{F\in R_\bullet}\cind_{P_{F}^\dagger}^G\delta_F,C^{i-1}(G,V))\\
		&= \Hom_{L}(\bigoplus_{F\in R_\bullet}\cind_{P_{F}^\dagger}^G\delta_F,C^{i-1}(G,V)).
	\end{align*}
	But now, the complex \[ 0\rightarrow \Hom_{L}(L,C^{i-1}(G,V))\rightarrow \Hom_{L}(\bigoplus_{F\in R_0}\cind_{P_{F}^\dagger}^G\delta_F,C^{i-1}(G,V)) \rightarrow \dots\]is exact, as we have taken $ L $-linear homomorphisms of an exact sequence into the $ L $-vector space $ C^{i-1}(G,V) $. Therefore, taking cohomology of the rows $ C^{\bullet,i} $ gives \begin{equation*}
		H^j(C^{\bullet,i})=\begin{cases}
		0,\text{ if }j\neq 0,\\
		\Hom_L(L,C^{i-1}(G,V))\cong C^{i-1}(G,V)\cong C^i(G,V)^G \text{, if }j=0.
		\end{cases}
	\end{equation*}
	The first page terms of the spectral sequence coming from the horizontal filtration are precisely given by these cohomology groups: \[ E_{h,1}^{i,j}=H^j(C^{\bullet,i}). \]
	Moreover, the first page differentials $ d_{h,1}^{i,j}:E_{h,1}^{i,j}\rightarrow  E_{h,1}^{i+1,j}$ are zero whenever $ j $ is non-zero and for $ j=0 $, they are just the maps\[ d_{h,1}^{i,0}:C^i(G,V)^G\rightarrow C^{i+1}(G,V)^G, \] induced by the differential maps on the complex $ C^\bullet(G,V) $. We can easily compute the second page terms: \[ E_{h,2}^{i,j}=H^{i,j}(E_{h,1}^{\bullet, \bullet})=\begin{cases}
	0\text{, if }j\neq 0 \\
	H^i(C^\bullet(G,V)^G)=H^i(G,V)\text{, if } j=0.
	\end{cases} \]
	In conclusion, we know that the spectral sequence coming from the horizontal filtration converges to the continuous cohomology groups of $ G $ with coefficients in $ V $. The spectral sequence induced by the vertical filtration also converges to $ H^\bullet(G,V) $, since it has the same limit term.
	
	On the other hand, we can fix an $ i\ge 0 $ and consider the columns of the double complex  \[ C^{i,\bullet} =\Hom_{G}(\bigoplus_{F\in R_i}\cind_{P_{F}^\dagger}^G\delta_F,C^\bullet(G,V)).  \]By pulling out the direct sum and applying Frobenius reciprocity for the compact induction, we obtain \begin{align}
	C^{i,\bullet}\cong &\bigoplus_{F\in R_i}\Hom_{P_F^\dagger}(\delta_F,C^\bullet(G,V))\notag\\
	\cong&\bigoplus_{F\in R_i}\Hom_{P_F^\dagger}(\1,C^\bullet(G,V)\otimes_L  \delta_F^{-1} ) . \label{column of double complex}
	\end{align}

	\begin{lem} \label{C(G,V) twisted}
		Let $ V $ be a topological $ G $-module on an $ L $-vector space, let $ H\le G $ be an open subgroup and $ \delta:H\rightarrow L^\times $ a smooth character on $ H $. We have an $ L $-linear isomorphism of topological $ H $-modules \[ C^n(G,V)\otimes_{L}\delta\overset{\cong}{\longrightarrow}C^n(G,V\otimes_{L}\delta),\ \forall n\ge 0. \]
	\end{lem}
	\begin{proof}
		As in Lemma \ref{Cn=C0}, we have homeomorphisms $ C^{n+1}(G,W) \rightarrow C^0(G,C^n(G,W)) $ for every topological space $ W $. These homeomorphisms are $ H $-equivariant if $ W $ is equipped with the structure of a topological $ H $-module. Hence, we may assume that $ n=0 $.
		
		Let $ \widetilde{\delta}:G\rightarrow L^\times$ be the map that sends an element $ g\in G $ to $ \delta(g) $ if $ g $ lies in $ H $ and to $ 1 $ otherwise. Since $ H $ is open in $ G $, this map is locally constant (but not necessarily a character). Therefore the map \[ \alpha: \ C^0(G,V)\otimes_{L}\delta\rightarrow C^0(G,V\otimes_{L}\delta),\] where \begin{equation*}
	 	\alpha(f\otimes 1)(g):=f(g)\otimes \widetilde{\delta}(g)
		\end{equation*}is a well-defined $ L $-linear isomorphism. Moreover, it is $ H $-equivariant, because for any $ h\in H $ and $ g\in G $ we have \begin{align*}
		\alpha(h(f\otimes 1))(g)&=\alpha((hf)\otimes \delta(h))(g)= f(gh)\otimes \widetilde{\delta}(g)\delta(h)\\
		 &=f(gh)\otimes \widetilde{\delta}(gh)=\alpha(f\otimes 1)(gh)=h\alpha(f\otimes 1 )(g).
		\end{align*}
		
		Continuity can be checked on a base of the compact-open topology on the right hand side. Such a base is given by the sets $ \Omega(K,U)=\{f\in C^0(G,V\otimes_{L}\delta) \vert \ f(K)\subset U \} $ for compact subsets $ K\subset G $ and open subsets $ U\subset V\otimes_{L}\delta$. Fixing such a $ K $ and $ U $, we see that $ \alpha^{-1}(\Omega(K,U)) $ consists of functions $ f $ such that for any $ k\in K $, $ \widetilde{\delta}(k)f(k) $ lies in the open subset $ U $. But since $ \widetilde{\delta} $ is locally constant and $ K $ compact, we can cover $ K $ by finitely many open sets of the form $ \widetilde{\delta}^{-1}(\lambda_i) $ for some $ \lambda_i\in L $, $ i=1,\dots,s $. Since the intersections $ K\cap \widetilde{\delta}^{-1} (\lambda_i)$ are again compact, we obtain that $  \alpha^{-1}(\Omega(K,U))$ is the finite intersection \[  \alpha^{-1}(\Omega(K,U))=\bigcap_{i=1}^s\Omega(K_i,\lambda_i^{-1}U)  \] of open sets and therefore is open itself. Thus, $ \alpha $ is continuous.
	\end{proof}
	
	We apply Lemma \ref{C(G,V) twisted} to the character $ \delta_F^{-1} $ in \eqref{column of double complex}, and obtain 	\begin{align*}
	C^{i,\bullet}\cong& \bigoplus_{F\in R_i}\Hom_{P_F^\dagger}(\1,C^\bullet(G,V\otimes_L  \delta_F^{-1})) \\
	\cong & \bigoplus_{F\in R_i}C^\bullet(G,V\otimes_L  \delta_F^{-1})^{P_F^\dagger}.
	\end{align*}
	As we have seen in Lemma \ref{computes H(K)}, the complex $ C^\bullet(G,V) $ is acyclic for the continuous cohomology of $ H $ for any open subgroup $ H\le G $ and hence, it can be used to compute the cohomology groups of $ H $. Therefore, taking the cohomology of this complex gives us the following: \begin{align*}
		H^j(C^{i,\bullet})&=\bigoplus_{F\in R_i}H^j(C^\bullet(G,V\otimes_L  \delta_F^{-1})^{P_F^\dagger})\\
		&=\bigoplus_{F\in R_i}H^j(P_F^\dagger, V\otimes_L\delta_F^{-1}).
	\end{align*}
	
	We thus get a spectral sequence, converging to the continuous cohomology of $ G $ with coefficients in $ V $, where the first page terms are given by the cohomology of the columns $ C^{i,\bullet} $ of our double complex. More precisely, we obtain the spectral sequence \eqref{spectral sequence general}.
	 
	 \subsection{Locally analytic vectors of admissible Banach space representations}
	 
	 The spectral sequence \eqref{spectral sequence general} can be very useful to generalize results that are known for compact groups to arbitrary $ p $-adic reductive groups. In this section, we will show an application of this kind, by comparing the continuous cohomology with coefficients in an admissible Banach space representation to the continuous cohomology with coefficients in its subrepresentation given by locally analytic vectors. We start with a short reminder on the definition of locally analytic representations. For more details on this matter, see for example Section 3 in \cite{emerton-locallyanalytic}. %\cite{schneider2004continuous}. %references?
	 
	 Let $ V $ be a barrelled locally convex Hausdorff $ L $-vector space, meaning that it is a Hausdorff topological  $ L $-vector space whose topology is defined by a family of seminorms $ q_i:V\rightarrow \mathbb{R}_{\ge 0} $, $ i\in I $, in the sense that the topology on $ V $ is the coarsest topology, making all seminorms continuous maps, and in which every closed lattice is open (for more details, see Chapter 1, Section 6 in \cite{schneider2001nonarchimedean}). For such a space, we denote by $ C^\la(G,V) $ the set of locally analytic $ V $-valued functions on $ G $.%barrelled
	 
	 \begin{defi}
	 	A \textit{locally analytic representation} of $ G $ over $ L $ is a representation of $ G $ on a barrelled locally convex Hausdorff $ L $-vector space, such that $ G $ is acting on $ V $ via continuous endomorphisms and for each $ v\in V $, the orbit map $ \rho_v:G\rightarrow V $, $ g\mapsto gv $, is a locally analytic function $ \rho_v\in C^\la(G,V) $.
	 \end{defi}
	 
	 Let $ \Pi $ be an admissible $ L $-Banach space representation of $ G $. By definition, $ \Pi $ is a normed $ L $-vector space, making it in particular a locally convex $ L $-vector space, so that it makes sense to talk about the locally analytic vectors in $ \Pi $, meaning the vectors $ v\in \Pi $, whose orbit maps $ \rho_v $ are locally analytic functions. The locally analytic vectors form a locally analytic subrepresentation of $ \Pi $, denoted by $ \Pi^\la $. The subrepresentation $ \Pi^\la $ is equipped with a topology as described in \cite{emerton-locallyanalytic}, making it a topological $ G $-module. The topology on $ \Pi^\la $ is finer than the subspace topology coming from $ \Pi^\la\subset \Pi $, so that the injection $ \Pi^\la\hookrightarrow \Pi $ is a continuous map. An explanation for this is given for example in the beginning of Section 3.5 in \cite{emerton-locallyanalytic}. 
	
	 By Corollary 1.6 in \cite{rodrigues2022solid}, for a compact $ p $-adic Lie group $ K $, we have isomorphisms \[ H^i(K,\Pi^\la)\cong H^i(K,\Pi). \]
	 Note that in \cite{rodrigues2022solid}, the condition for the existence of such isomorphism is that the representation $ \Pi $ has no higher locally analytic vectors, which is true for admissible Banach space representations (cf. Theorem 7.1 \cite{schneider2002algebras}).
	 
	 We want to use the spectral sequence \eqref{spectral sequence general} to generalize this to the following:
	 
	 \begin{prop} \label{prop: adm and la same}
	 	Let $ \Pi $ be an admissible $ L $-Banach space representation of $ G $. Then for every $ i\ge 0 $, the maps\[ H^i(G,\Pi^\la)\overset{\cong}{\longrightarrow} H^i(G,\Pi), \]
	 	induced by the inclusion $ \Pi^\la \hookrightarrow \Pi $, are isomorphisms.
	 \end{prop}
	 
	 Before we can apply the spectral sequence argument, recall that the stabilizer groups $ P_F^\dagger $ appearing in the sequence are not compact. We therefore need the following lemmas first.
	 
	 \begin{lem}\label{lemma cohomology of exact functor}
	 	Let $ \Gamma $ be a discrete group and let $ 0\rightarrow M^0\overset{d^0}{\rightarrow} M^1 \overset{d^1}{\rightarrow} M^2\rightarrow \dots $ be a complex of $ \Gamma $-modules. Then given an exact functor $ F:\MOD_\Gamma\rightarrow \MOD_\Z $, we have isomorphisms \[ H^i(F(M^\bullet))\cong F(H^i(M^\bullet)) \] for every $ i\ge 0 $.
	 \end{lem}
 
 	\begin{proof}
 		We fix the following notation: for any complex \[ 0\rightarrow C^0\overset{\delta^0}{\rightarrow}C^1 \overset{\delta^1}{\rightarrow}C^2\rightarrow \dots, \]
 		let $ Z^i(C^\bullet):=\ker(\delta^i) $ be the $ i $-cocycles and $ B^i(C^\bullet):=\image(\delta^{i-1})\cong C^{i-1}/Z^{i-1}(C^\bullet)  $ the $ i $-coboundaries. 
 		
 		Applying the exact functor $ F $ to the exact sequence \[ 0\rightarrow Z^i(M^\bullet)\rightarrow M^i\rightarrow M^{i+1} \] gives the isomorphism \begin{equation}\label{eq ZiF=FZi}
 		Z^i(F(M^\bullet))=F(Z^i(M^\bullet)),\ \forall i\ge 0. 
 		\end{equation}
 		
 		Using this identification \eqref{eq ZiF=FZi} and the exactness of $ F $, we can also rewrite the coboundaries as  \begin{equation}\label{eq BiF=FBi}
 		B^i(F(M^\bullet))\cong \frac{F(M^{i-1})}{Z^{i-1}(F(M^\bullet))}\cong \frac{F(M^{i-1})}{F(Z^{i-1}(M^\bullet))}\cong F\left( \frac{M^{i-1}}{Z^{i-1}(M^\bullet)}\right) \cong F(B^i(M^\bullet)) .
 		\end{equation}
 		
 		Combining these isomorphisms \eqref{eq ZiF=FZi} and \eqref{eq BiF=FBi} gives us  \begin{equation*}
 		H^i(F(M^\bullet))\cong \frac{F(Z^i(M^\bullet))}{F(B^i(M^\bullet))}\cong F\left( \frac{Z^i(M^\bullet)}{B^i(M^\bullet)}\right) \cong F(H^i(M^\bullet)),
 		\end{equation*}
 		as claimed.
 	\end{proof}
	 	
	 \begin{lem}\label{lemma discrete group cohomology commutes}
	 	Let $ \Gamma $ be a discrete group, then for every $ i\ge 0 $, the functor $ C^i(\Gamma,-)^\Gamma $ is exact.
	 	
	 	In particular, for any complex $ 0\rightarrow M^0\overset{d^0}{\rightarrow} M^1 \overset{d^1}{\rightarrow} M^2\rightarrow \dots$ of $ \Gamma $-modules, we have \[ H^j(C^i(\Gamma, M^\bullet)^\Gamma)\cong C^i(\Gamma,H^j(M^\bullet))^\Gamma, \ \forall i,j\ge 0. \]
	 \end{lem}
 
 	\begin{proof}
 		To show that the functor $ C^i(\Gamma,-)^\Gamma $ is exact, let \[ 0\rightarrow N \rightarrow M\rightarrow Q\rightarrow 0 \] be a short exact sequence of $ \Gamma $-modules. Note that, since $ \Gamma $ is discrete, any function on $ \Gamma $ (or on a product $ \Gamma^{i+1} $ of copies of $ \Gamma$) is automatically continuous and hence, the functor $ C^i(\Gamma,-) $ is exact and we obtain a short exact sequence \[  0\rightarrow C^i(\Gamma,N) \rightarrow C^i(\Gamma,M)\rightarrow C^i(\Gamma,Q)\rightarrow 0. \]
 		
 		Applying the functor $ (-)^\Gamma $ to this, induces a long exact sequence in cohomology \[  0\rightarrow C^i(\Gamma,N)^\Gamma \rightarrow C^i(\Gamma,M)^\Gamma\rightarrow C^i(\Gamma,Q)^\Gamma\rightarrow H^1(\Gamma,C^i(\Gamma,N)),\] where the first cohomology group $ H^1(\Gamma,C^i(\Gamma,N))$ vanishes, because $ C^i(\Gamma,N) $ is acyclic, since by Lemma \ref{Cn=C0} and Shapiro's Lemma \cite[Proposition 3]{CW}, we have \[  H^i(\Gamma,C^i(\Gamma,N))\cong H^i(\Gamma,\Ind_{\{1\}}^\Gamma C^{i-1}(\Gamma,N))\cong H^i(\{1\},C^{i-1}(\Gamma,N))=0 \] for $ i>0 $. This proves exactness of $ C^i(\Gamma,-)^\Gamma $.
 		
 		We can thus apply Lemma \ref{lemma cohomology of exact functor} to the functor $ C^i(\Gamma,-)^\Gamma $ and the complex $ M^\bullet $, to obtain $H^j(C^i(\Gamma, M^\bullet)^\Gamma)\cong C^i(\Gamma,H^j(M^\bullet))^\Gamma  $.
 	\end{proof}

	 \begin{lem}\label{lem hochschild-serre}
	 	Let $ V $ be a topological $ G $-module and let $ H\le G $ be an open normal subgroup of $ G $. Then there exists a Hochschild--Serre spectral sequence:\[ E_2^{i,j}\cong H^i(G/H,H^j(H,V))\Rightarrow H^{i+j}(G,V). \]
 	 \end{lem}
  
  	\begin{proof}
  		We mimic the proof of Proposition 5 in \cite{CW}, where the authors prove the existence of a Hochschild--Serre spectral sequence for a closed normal subgroup $ H $, assuming some additional conditions that are not needed for the case of an open normal subgroup. We consider the double complex \[ C^{i,j}=C^i(G/H,C^j(G,V)^H)^{G/H},\ i,j\ge 0. \]
  		This double complex induces two spectral sequences with the same limit term. The first spectral sequence comes from the horizontal filtration of the double complex and its $ E_1 $-terms are given by:
  		\begin{align*}
	  		E_{1,h}^{i,j}=H^j(C^{\bullet,i})=H^j(C^\bullet(G/H,C^i(G,V)^H)^{G/H})=H^j(G/H,C^i(G,V)^H)
  		\end{align*}
  		Note that the groups $ C^i(G,V)^H $ are acyclic for the cohomology of $ G/H $. Indeed, since $ G/H $ is discrete, the continuous cohomology of $ G/H $ is just group cohomology. Now for any $ G/H $-module $ W $, we can give it a topological $ G/H $-module structure by equipping it with the discrete topology and obtain \begin{align*}
	  		\Hom_{G/H}(W,C^i(G,V)^H)&=\Hom^{\cts}_{G/H}(W,C^i(G,V)^H)\\&\cong \Hom^{\cts}_G(W,C^i(G,V))\\&\cong \Hom^{\cts}(W,C^{i-1}(G,V))\\&=\Hom(W,C^{i-1}(G,V)),
  		\end{align*}
  		using Lemma \ref{Frobenius}. Since the functor $ \Hom(-,C^{i-1}(G,V)) $ is exact, so is the functor $ \Hom_{G/H}(-,C^i(G,V)^H) $, showing that $ C^i(G,V)^H $ is acyclic for the cohomology of $ G/H $.
  		
  		Hence, we obtain \[  E_{1,h}^{i,j}=H^j(G/H,C^i(G,V)^H)=\begin{cases}
	  		0\text{, if }j\neq 0\\
	  		(C^i(G,V)^H)^{G/H}\cong C^i(G,V)^G\text{, if }j=0.
  		\end{cases}\]
  		And the $ E_2 $-terms are given by \[ E_{2,h}^{i,j}=\begin{cases}
	  		0\text{, if }j\neq 0\\
	  		H^i(C^\bullet(G,V)^G)=H^i(G,V)\text{, if }j=0.
  		\end{cases} \]
  		Therefore, the two spectral sequences of the double complex $ C^{i,j} $ converge to the continuous cohomology of $ G $ with coefficients in $ V $.
  		
  		The $ E_1 $-terms of the spectral sequence coming from the vertical filtration are given by \begin{align*}
  		E_{1,v}^{i,j}=H^j(C^{i,\bullet})=H^j(C^i(G/H,C^\bullet(G,V)^H)^{G/H}).
  		\end{align*}
  		
  		We can apply Lemma \ref{lemma discrete group cohomology commutes} to the discrete group $ \Gamma=G/H $ and the complex $ M^\bullet=C^\bullet(G,V)^H $ and get \[ H^j(C^i(G/H,C^\bullet(G,V)^H)^{G/H})\cong C^i(G/H,H^j(C^\bullet(G,V)^H))^{G/H}. \]
  		
  		By Lemma \ref{computes H(K)}, the resolution $ C^\bullet(G,V) $ can be used to compute the continuous cohomology of $ H $, so that \[ E_{1,v}^{i,j}=C^i(G/H,H^j(C^\bullet(G,V)^H))^{G/H}=C^i(G/H,H^j(H,V))^{G/H}. \]
  		
  		Then the $ E_2 $-terms are given by \[ E_{2,v}^{i,j}=H^{i,j}(E_{1,v}^{\bullet,\bullet})=H^i(C^\bullet(G/H,H^j(H,V))^{G/H}) =H^i(G/H,H^j(H,V)).\]
  		
  		Since we know that this spectral sequence has the same limit term as the one from the horizontal filtration, we obtain the wanted spectral sequence.
  	\end{proof}
 
	 \begin{lem} \label{lem. H(P_F,la)}
	 	Let $ F\in X_q $ be a $ q $-dimensional facet of the Bruhat--Tits building of $ G $. Then for every $ i\ge 0 $, we have isomorphisms \[ H^i(P_F^\dagger,\Pi^\la)\cong H^i(P_F^\dagger,\Pi). \]
	 \end{lem}
 
 	\begin{proof}
 		Let us fix a facet $ F\in X_q $. Recall that there exists a compact open subgroup $ R_F\le G $, which is open and normal in $ P_F^\dagger $ (for a construction, see Section 1.2 in \cite{schneider1997representation}).
 		
 			Since the subgroup $ R_F $ is compact open in $ P_F^\dagger $, we can apply Lemma \ref{lem hochschild-serre}, which proves the existence of a Hochschild--Serre spectral sequence \[ E_2^{i,j}=H^i(P_F^\dagger/R_F, H^j(R_F,\Pi))\Rightarrow H^{i+j}(P_F^\dagger,\Pi). \]
 		 		By the same argument, we also obtain a spectral sequence for the locally analytic subrepresentation \[ \widetilde{E}_2^{i,j}=H^i(P_F^\dagger/R_F, H^j(R_F,\Pi^\la))\Rightarrow H^{i+j }(P_F^\dagger,\Pi^\la). \]
 		 		On the other hand, the inclusion $ \Pi^\la\hookrightarrow \Pi $ induces maps between the limit terms $ H^i(P_F^\dagger,\Pi^\la) \rightarrow H^i(P_F^\dagger,\Pi) $, as well as a map between all second page terms of the spectral sequence \[  \widetilde{E}_2^{i,j}=H^i(P_F^\dagger/R_F, H^j(R_F,\Pi^\la))\rightarrow H^i(P_F^\dagger/R_F, H^j(R_F,\Pi))=E_2^{i,j} .\]
 		 		We thus obtain a map of spectral sequences $   \widetilde{E}_r^{i,j}\rightarrow E_r^{i,j} . $

 		But since $ R_F $ is compact, Corollary 1.6 in \cite{rodrigues2022solid} tells us that these are isomorphisms for $ r=2 $ and for all $ i , j\ge 0 $. By the Comparison Theorem 5.2.12 in \cite{weibel1995introduction}, this implies that also the map of the limit terms \[ H^i(P_F^\dagger,\Pi^\la) \rightarrow H^i(P_F^\dagger,\Pi) \]
 		is an isomorphism.
 	\end{proof}
 
 	We can now prove Proposition \ref{prop: adm and la same}.
 	\begin{proof}[Proof of Proposition \ref{prop: adm and la same}]
 		The argumentation is similar to the one in the proof of Lemma \ref{lem. H(P_F,la)}. By the discussion in Section \ref{subsec spectral sequ}, we obtain spectral sequences like in \eqref{spectral sequence general} for both, $ \Pi $ and $ \Pi^\la $: \begin{align*}
	 		\widetilde{E}_1^{i,j}=&\bigoplus_{F\in R_i}H^j(P_F^\dagger, \Pi^\la\otimes_L \delta_F\inv)\Rightarrow H^{i+j}(G,\Pi^\la),\\
	 		 E_1^{i,j}=&\bigoplus_{F\in R_i}H^j(P_F^\dagger, \Pi\otimes_L \delta_F\inv)\Rightarrow H^{i+j}(G,\Pi).
 		\end{align*}
 		But again, the inclusion $ \Pi^\la \hookrightarrow \Pi $ induces maps between all the terms. And by Lemma \ref{lem. H(P_F,la)}, those maps between the $ E_1 $ terms are isomorphisms. Theorem 5.2.12 in \cite{weibel1995introduction} implies that the maps \[ H^i(G,\Pi^\la)\rightarrow H^i(G,\Pi) \] are isomorphisms.
 	\end{proof}  
 
% 	Combining Proposition \ref{prop: adm and la same} with Corollary \ref{H^i(G,Pi) are finite diml}, the following Corollary follows immediately.
 	
 	\begin{cor} \label{cor H(G,Pi^la) finite}
 		For any admissible $ L $-Banach space representation $ \Pi $ of $ G $, the cohomology groups $ H^i(G,\Pi^\la) $ with coefficients in the locally analytic vectors of $ \Pi $ are finite dimensional over $ L $.
 	\end{cor}
 
 	\begin{proof}
 		This follows immediately from Proposition \ref{prop: adm and la same} combined with the fact that the continuous cohomology groups $ H^i(G,\Pi) $ for an admissible $ L $-Banach space representation $ \Pi $ is finite dimensional (cf. Corollary 6.10 in \cite{fust2021}).
 	\end{proof}
	 
	 \subsection{Topology on the cohomology groups}
	 
	 Let $ V $ be a Hausdorff topological $ G $-module. Note that the continuous cohomology groups $ H^i(G,V) $ inherit a quotient topology from the compact-open topology on the cochains. We call this the \emph{canonical topology}. The compact-open topology on the cochains $ C^i(G,V) $ is Hausdorff, since for any pair of continuous maps $ f,f':G^{i+1}\rightarrow V $ in $ C^i(G,V) $ with $ f\neq f' $, we can find disjoint open neighborhoods of $ f $, $ f' $, respectively, as follows: There is an element $ g\in G^{i+1} $ such that $ f(g)\neq f'(g) $ and since $ V $ is Hausdorff by assumption, we can find open disjoint neighborhoods $ U\ni f(g) $, $ U'\ni f'(g) $, $ U\cap U'=\emptyset$. But then we have open disjoint neighborhoods $ \Omega(\{g\},U)\ni f $ and $ \Omega(\{g\},U') \ni f'$.
	 
	 Since $ C^i(G,V) $ is Hausdorff, so is $ C^i(G,V)^G $ as a subspace of a Hausdorff space. And the space of $ i $-cocycles $ Z^i(G,V)=\ker(d^i:C^i(G,V)^G\rightarrow C^{i+1}(G,V)^G) $ is also Hausdorff and moreover, it is a closed subspace of $ C^i(G,V)^G $ as it is the kernel of a continuous map between Hausdorff spaces. This also implies that the quotient topology on the $ i $-coboundaries given by $ B^i(G,V)\cong C^{i-1}(G,V)^G/Z^{i-1}(G,V) $ is Hausdorff.
	 
	 The quotient topology on the cohomology groups is in general not Hausdorff. However, Proposition 6 in \cite{CW} gives a nice criterion for the cohomology groups being Hausdorff, which almost fits our setting. We need to assume though that the group $ G $ is $ \sigma $-compact, which means that it can be written as a countable union of compact subsets. If $ F $ is a finite extension of $ \Q_p $, then the group $ \GL_n(F) $ is $ \sigma $-compact. This can be shown by using the Cartan decomposition (cf. Section 3.2 in \cite{bernshtein1976representations}). 
	 
	 As a consequence, we know that any $ p $-adic reductive group $ G=\mathbb{G}(F) $, where $ \mathbb{G} $ is a linear algebraic group, is also $ \sigma $-compact, since it embeds into $ \GL_n(F) $ as a closed subgroup for some $ n $. We then obtain the following:
	 
	 \begin{cor}\label{cor Hausdorff}
	 	Assume that $ G $ is a $ p $-adic reductive group and $ \Pi $ is an admissible $ L $-Banach space representation of $ G $. Then the canonical topologies on the cohomology groups $ H^i(G,\Pi)$ and $H^i(G,\Pi^\la) $ are Hausdorff for all $ i\ge 0 $.
	 \end{cor}
 
 	\begin{proof}
 		Corollary 6.10 in \cite{fust2021} implies that the cohomology groups are finite dimensional. Moreover, since $ \Pi $ is assumed to be an $ L $-Banach space representation, it is a complete metrizable $ L $-vector space. We can hence apply Proposition 6 in \cite{CW}, which states that the cohomology groups are strongly Hausdorff, i.e.~the inclusion of $ i $-coboundaries $ B^i(G,\Pi) $ into $ i $-cocycles $ Z^i(G,\Pi) $ has a continuous section, say $ s:Z^i(G,\Pi)\rightarrow B^i(G,\Pi) $. We claim that the continuous $ L $-linear bijection \begin{align*}
	 		\ker(s)\hookrightarrow Z^i(G,\Pi)\overset{\pi}{\rightarrow} H^i(G,\Pi)
 		\end{align*}
 		is a homeomorphism. It suffices to show that the map is open. For this, let $ U\subset \ker(s) $ be an open subset. This means that there is an open subset $ V\subset Z^i(G,\Pi) $ such that $ U=V\cap \ker(s) $. Since the map $ \id - s:Z^i(G,\Pi)\rightarrow Z^i(G,\Pi) $ is continuous, the set $ V':=(\id-s)\inv(V)=\{z\in Z^i(G,\Pi)\vert z-s(z)\in V\} $ is open in $ Z^i(G,\Pi) $. Moreover, the intersection  \begin{align*}
 		V'\cap \ker(s)&=\{z \in Z^i(G,\Pi)\vert z-s(z)\in V \text{ and }s(z)=0\}  \\
 		&=\{z \in Z^i(G,\Pi)\vert s(z)=0\text{ and }z\in V\} \\&=V\cap \ker(s)
 		\end{align*}coincides with the intersection $ V\cap \ker(s)=U $. But then we have \[ \pi(U)=\pi(V\cap \ker(s))=\pi(V'\cap \ker(s))=\pi(V')\text{ open,} \]
 		since for any $ z\in V' $, $ \pi(z)=\pi(z)-\pi(s(z))=\pi(z-s(z)) $ with $ z-s(z)\in V\cap \ker(s) $.
 		
 		As the kernel of a continuous map between Hausdorff spaces, $ \ker(s) $ is Hausdorff and thus so is $ H^i(G,\Pi) $.
 		
 		By Proposition \ref{prop: adm and la same}, the continuous $ L $-linear map $ H^i(G,\Pi^\la)\rightarrow H^i(G,\Pi) $ is a bijection. This implies that the topology on $ H^i(G,\Pi^\la) $ is finer than (or equal to) the topology on $ H^i(G,\Pi) $ which is Hausdorff, so that the groups $ H^i(G,\Pi^\la) $ are also Hausdorff.
 	\end{proof}
 
 	\begin{lem} \label{lem homeo check on subbasis}
 		Let $ V $ and $ W $ be topological $ L $-vector spaces and $ \phi:V\overset{\cong}{\rightarrow}W $ be an $ L $-linear isomorphism. Assume that $ \mathcal{B} $ is a subbasis of open neighborhoods of $ 0\in V $, such that the set $ \phi(\mathcal{B})=\{\phi(U)\vert\ U\in \mathcal{B} \} $ is a subbasis of open neighborhoods of $0\in W $. Then the isomorphism $ \phi $ is a homeomorphism.
 	\end{lem}
 
 	\begin{proof}
 		By symmetry, it is enough to show that the map $ \phi $ is continuous. For this, let $ U\subset W $ be an open set. We show that $ \phi\inv(U) $ is open in $ V $, by showing that each element $ v\in \phi\inv(U) $ has an open neighborhood that is contained in $ \phi\inv(U) $.
 		
 		Since $ W $ is a topological vector space, the translated set $ -\phi(v)+U $ is again open in $ W $ and contains the zero element. By assumption, we can find finitely many open neighborhoods of $ 0 $ in $ V $, $ V_1,\dots,V_n \in \mathcal{B}$, such that the intersection $ \bigcap_{i=1}^n\phi(V_i) $ is contained in $ -\phi(v)+U $. This implies that the intersection of open sets $ \bigcap_{i=1}^n(v+V_i) $ is contained in the inverse image $ \phi\inv(U) $ and contains $ v $. Therefore, the map is indeed continuous. 
 	\end{proof}
 
 	\begin{prop}\label{prop C(G,V) locally convex}
 		Let $ V $ be a topological $ G $-module on a locally convex $ L $-vector space. Then the compact-open topology on the space of cochains $ C^i(G,V) $ is locally convex, for every $ i\ge 0 $.
 	\end{prop}
 
 	\begin{proof}
 		By the heomeomorphism $ C^i(G,V)\cong C^0(G,C^{i-1}(G,V)) $ (cf. Lemma \ref{Cn=C0}), it is enough to show that $ C(G,V)=C^0(G,V) $ is locally convex.
 		
 		Say the locally convex topology on $ V $ is defined by the family of seminorms $ \{q_i\}_{i\in I} $. For every compact subset $ K\subset G $, and $ i\in I $, we define a seminorm on $ C(G,V) $ by \begin{align*}
 		q_{K,i}:C(G,V)&\rightarrow \mathbb{R},\\
 		f&\mapsto \sup_{k\in K} q_i(f(k))
 		\end{align*}
 		Denote by $ C(G,V)_{\lc} $ the space $ C(G,V) $ equipped with the locally convex topology given by the family $ \{q_{K,i}\}_{K,i} $ of seminorms and by $ C(G,V)_{\co} $ the space $ C(G,V) $ equipped with the compact-open topology. We claim that the identity map $ C(G,V)_{\co}\rightarrow C(G,V)_{\lc} $ is a homeomorphism. 
 		
 		By Lemma \ref{lem homeo check on subbasis}, it is enough to find a set of neighborhoods of $ 0\in C(G,V) $ that forms a subbasis of open neighborhoods in both topologies. By definition, a subbasis of open neighborhoods of $ 0\in C(G,V)_{\lc} $ is given by sets of the form $ q_{K,i}\inv(B_\epsilon) $, for $ K\subset G $ compact, $ i\in I $, $ \epsilon >0 $. Here, $ B_\epsilon $ denotes the open ball of radius $ \epsilon  $ around $ 0\in \mathbb{R} $. We can rewrite these sets as follows: \begin{align*}
 		q_{K,i}\inv(B_\epsilon) &=\{ f\in C(G,V)\vert\ q_{K,i}(f) <\epsilon \}\\
 		&= \{ f\in C(G,V) \vert\ \sup_{k\in K} q_i(f(k)) <\epsilon \}\\
 		&= \{ f\in C(G,V)\vert\ q_i(f(k)) <\epsilon,\ \forall k\in K \}\\
 		&= \{ f\in C(G,V)\vert\ f(k)\in q_i\inv(B_\epsilon)\ \forall k\in K \}\\
 		&=\Omega(K,q_i\inv(B_\epsilon)).
 		\end{align*}
 		These sets are open neighborhoods of $ 0 $ in $ C(G,V)_{\co} $. We want to show that they form a subbasis of open neighborhoods of $ 0 $. Indeed, by definition of the compact-open topology, a subbasis of open neighborhoods of $ 0 $ is given by sets of the form $ \Omega(K,U) $ for $ K\subset G $ compact and $ U\subset V $ open. But since $ V $ is locally convex, we can find finitely many open sets of the form $ q_{i_j}\inv(B_{\epsilon_j}) $, for $ j=1,\dots, m $, with the property that \[ \bigcap_{j=1}^mq_{i_j}\inv(B_{\epsilon_j})\subset U .\]
 		We thus obtain \[ \bigcap_{j=1}^m\Omega(K,q_{i_j}\inv(B_{\epsilon_j}))\subset \Omega(K,U), \]
 		and we see that the sets $ \Omega(K,q_i\inv(B_\epsilon)) $ form a subbasis of open neighborhoods of $ 0 $ in $ C(G,V)_{\co} $ as stated.
 	\end{proof}
 	
 	Since subspaces and quotients of locally convex vector spaces are again locally convex (see Section 5 in \cite{schneider2001nonarchimedean}), Proposition \ref{prop C(G,V) locally convex} implies the following corollary.
 
 	\begin{cor}\label{cor H(G,V) locally convex}
 		For every topological $ G $-module on a locally convex $ L $-vector space, the canonical topology on the continuous cohomology groups $ H^i(G,V) $ is locally convex.
 	\end{cor}
 
 	\begin{cor}\label{cor finest locally convex}
 		For every admissible $ L $-Banach space representation $ \Pi $ of $ G $, the continuous cohomology groups $ H^i(G,\Pi) $ and $ H^i(G,\Pi^\la) $, equipped with their canonical topologies, are finite dimensional locally convex Hausdorff $ L $-vector spaces. In particular, their topology is uniquely determined, namely it is the finest locally convex topology. 
 		
 		In particular, the isomorphisms $ H^i(G,\Pi^\la)\cong H^i(G,\Pi) $ are homeomorphisms.
 	\end{cor}
 
 	\begin{proof}
 		By Corollary 6.10 in \cite{fust2021} and Corollaries \ref{cor Hausdorff} and \ref{cor H(G,V) locally convex}, we know that the cohomology groups are finite dimensional, Hausdorff and locally convex. Therefore, we can apply Proposition 4.13 in \cite{schneider2001nonarchimedean}, which proves that there is only one choice of locally convex Hausdorff topology on a finite dimensional $ L $-vector spaces. The author gives an explicit description of this topology as the one defined by the norm $\Vert \sum_{i=1}^n\lambda_i e_i\Vert:=\max_{1\le i \le n}\vert \lambda_i\vert  $, for any choice of $ L $-basis $ e_1,\dots,e_n $ of the vector space. 

 		Since the canonical topologies on $ H^i(G,\Pi^\la) $ and $ H^i(G,\Pi) $ are uniquely determined and the groups are isomorphic as $ L $-vector spaces (Proposition \ref{prop: adm and la same}), the topologies must agree and the isomorphism is a homeomorphism of topological $ L $-vector spaces.
 	\end{proof}

	\bibliographystyle{siam}
	\bibliography{thesis-test}

\begin{thebibliography}{10}

\bibitem{bernshtein1976representations}
{\sc I.~N. Bern\v{s}te\u{\i}n and A.~V. Zelevinski\u{\i}}, {\em Representations
  of the group {$GL(n,F),$} where {$F$} is a local non-{A}rchimedean field},
  Uspekhi Mat. Nauk, 31 (1976), pp.~5--70.

\bibitem{Bourbaki}
{\sc N.~Bourbaki}, {\em \'{E}l\'{e}ments de math\'{e}matique. {X}. {P}remi\`ere
  partie: {L}es structures fondamentales de l'analyse. {L}ivre {III}:
  {T}opologie g\'{e}n\'{e}rale. {C}hapitre {X}: {E}spaces fonctionnels;
  dictionnaire}, Actualit\'{e}s Scientifiques et Industrielles [Current
  Scientific and Industrial Topics], No. 1084, Hermann et Cie., Paris, 1949.

\bibitem{CW}
{\sc W.~Casselman and D.~Wigner}, {\em Continuous cohomology and a conjecture
  of {S}erre's}, Invent. Math., 25 (1974), pp.~199--211.

\bibitem{emerton-locallyanalytic}
{\sc M.~Emerton}, {\em Locally analytic vectors in representations of locally
  $p$-adic analytic groups}, vol.~248, American Mathematical Society, 2017.

\bibitem{fust2021}
{\sc P.~Fust}, {\em Continuous group cohomology and {E}xt-groups}, M\"unster
  Journal of Mathematics, 15 (2022), pp.~279--304.

\bibitem{rodrigues2022solid}
{\sc J.~Rodrigues~Jacinto and J.~Rodr{\'\i}guez~Camargo}, {\em Solid locally
  analytic representations of $p$-adic {L}ie groups}, Representation Theory of
  the American Mathematical Society, 26 (2022), pp.~962--1024.

\bibitem{schneider2001nonarchimedean}
{\sc P.~Schneider}, {\em Nonarchimedean functional analysis}, Springer Science
  \& Business Media, 2001.

\bibitem{schneider1997representation}
{\sc P.~Schneider and U.~Stuhler}, {\em Representation theory and sheaves on
  the {B}ruhat-{T}its building}, Inst. Hautes \'{E}tudes Sci. Publ. Math.,
  (1997), pp.~97--191.

\bibitem{schneider2002algebras}
{\sc P.~Schneider and J.~Teitelbaum}, {\em Algebras of p-adic distributions and
  admissible representations}, Inventiones mathematicae, 153 (2002),
  pp.~145--196.

\bibitem{weibel1995introduction}
{\sc C.~A. Weibel}, {\em An introduction to homological algebra}, vol.~38 of
  Cambridge Studies in Advanced Mathematics, Cambridge University Press,
  Cambridge, 1994.

\end{thebibliography}
\end{document}